\pdfoutput=1
\documentclass[final]{amsart}
\usepackage[utf8]{inputenc}
\usepackage{microtype}
\usepackage{amsmath}
\usepackage{mathtools}
\usepackage{amsfonts}
\usepackage{amssymb}
\usepackage{amsthm}
\usepackage{thmtools}
\usepackage{xcolor}
\usepackage{mathrsfs}
\usepackage{enumitem}
\usepackage{hyperref}
\usepackage[msc-links]{amsrefs}

\hypersetup{
  colorlinks,
  linkcolor=-red!75!green!50,
  citecolor=-red!75!green!50,
  urlcolor=green!30!black
}

\usepackage{leading}
\leading{13pt}

\makeatletter
\AtBeginDocument{%
  \expandafter\let\csname[\endcsname\relax
  \expandafter\let\csname]\endcsname\relax
  \expandafter\DeclareRobustCommand\csname[\expandafter\endcsname\expandafter{%
    \csname begin\endcsname{equation}%
  }%
  \expandafter\DeclareRobustCommand\csname]\expandafter\endcsname\expandafter{%
    \csname end\endcsname{equation}%
  }%
}
\makeatother
\mathtoolsset{showonlyrefs,showmanualtags}

\declaretheorem[numbered=no, name=Theorem]{Theorem-Intro}
\declaretheorem[numbered=no, name=Question]{Question-Intro}
\declaretheorem{Lemma}
\declaretheorem[sibling=Lemma]{Theorem}
\declaretheorem[sibling=Lemma]{Proposition}

\newlist{thmlist}{enumerate}{1}
\setlist[thmlist]{
        nolistsep,
        ref={\mdseries\textup{(\emph{\roman*})}},
        label={\mdseries\textup{(\emph{\roman*})}},
        }
\newcommand\thmitem[1]{\textup{(\emph{\romannumeral #1})}}

\makeatletter
\newlength{\@thlabel@width}
\newcommand{\fixhspace}{%
        \settowidth{\@thlabel@width}{\hskip0.5em\itshape1.}%
        \sbox{\@labels}{%
                \unhbox\@labels%
                \hspace{\dimexpr-\leftmargin+\labelsep+\@thlabel@width-\itemindent}
                }
        }
\makeatother

\renewcommand\O{\mathcal{O}}
\newcommand\II{\mathbb{I}}
\newcommand\NN{\mathbb{N}}
\newcommand\CC{\mathbb{C}}

\newcommand\kk{\Bbbk}
\newcommand\A{\mathcal{A}}
\DeclareMathOperator\Der{Der}
\DeclareMathOperator\GL{GL}
\DeclareMathOperator\Diff{Diff}
\DeclareMathOperator\perm{perm}
\DeclareMathOperator\Sym{Sym}

\DeclareMathOperator\gr{gr}
\newcommand\idx[1]{\mathbf{#1}}
\newcommand\B{\mathscr{B}}
\newcommand\subst[3]{#1\langle #2/#3\rangle}

\DeclarePairedDelimiter\abs{\lvert}{\rvert}

\colorlet{newterm-color}{blue!50!black}
\newcommand\newterm[1]{\textbf{\itshape\color{newterm-color}#1}}

\title[The algebra of differential operators]{The algebra of differential
operators tangent to a hyperplane arrangement}

\author{Mariano Suárez-Álvarez}
\thanks{IMAS (CONICET) -- Universidad de Buenos Aires}
\address{Departamento de Matem\'atica, Facultad de Ciencias Exactas y
Naturales, Universidad de Buenos Aires. Ciudad Universitaria,
Pabell\'on I (1428) Ciudad de Buenos Aires, Argentina.}
\email{mariano@dm.uba.ar}

\date{June 14, 2018}

\begin{document}

\maketitle

Let us fix a ground field~$\kk$, once and for all, and a vector space $V$
of positive dimension~$\ell$, and let $S$ be the graded algebra of
polynomial functions on~$V$. Let $\A$ be a central hyperplane arrangement
in~$V$, that is, a set of subspaces $H_1$,~\dots,~$H_r$ of~$V$ of
codimension~$1$, and for each $i\in\{1,\dots,\ell\}$ let $\alpha_i\in V^*$
be a linear form on~$V$ such that $H_i=\ker\alpha_i$. The product
$Q=\alpha_1\cdots\alpha_r\in S$, which we call a defining polynomial of the
arrangement,  has zero locus equal to the union~$N(\A)=\bigcup_{i=1}^rH_i$
of the hyperplanes of~$\A$ and depends, up to a non-zero scalar, only on
the arrangement.

We write $\Der(S)$ the Lie algebra of all derivations of~$S$ endowed with
its usual structure of a graded left $S$-module and for each ideal
$I\subseteq S$ we consider the set $\Der(I) =
\{\delta\in\Der(S):\delta(I)\subseteq I\}$, which is both a Lie subalgebra
and a graded $S$-submodule of~$\Der(S)$. The \newterm{Lie
algebra~$\Der(\A)$ of derivations tangent to~$\A$} is $\Der(QS)$ and the
arrangement is \newterm{free} if $\Der(\A)$ is a free graded $S$-module.
This notion was introduced in the context of the study of logarithmic
vector fields related to a hypersurface of a smooth complex analytic
manifold by Kyoji Saito in~\cite{Saito}. We do not know what exactly makes
an arrangement free and it is not a generic property, but we have plenty of
evidence that freeness is both important and useful, and it shows up in
many contexts. For example, Hiroaki Terao proved in~\cite{Terao} that if
$\kk$ is the field of complex numbers and $G$ is a finite subgroup
of~$\GL(V)$ generated by pseudo-reflections, then the arrangement $\A_G$ of
the reflecting hyperplanes of all pseudo-reflections in~$G$ is free.

\medskip

Recall that if $R$ is an algebra and $I$ is a right ideal of~$R$, the
\newterm{idealizer} of~$I$ in~$R$ is the subalgebra~$\II_R(I)=\{ r\in R:rI\subseteq
I\}$ of~$R$, easily seen to be the largest subalgebra of~$R$ that
contains~$I$ as a bilateral ideal. We use this notion in our situation as
follows. Let $\Diff(S)$ be the algebra of all regular differential
operators on~$S$ and let us view $S$ as a subalgebra of~$\Diff(S)$ by
identifying each $f\in S$ with the differential operator of order zero
given by multiplication by~$f$. For each $f\in S$ we denote
$\Diff(f^\infty\Diff(S))$ the intersection of idealizers
$\bigcap_{t\geq1}\II_{\Diff(S)}(f^t\Diff(S))$ and we define the
\newterm{algebra~$\Diff(\A)$ of differential operators tangent to the
arrangement~$\A$} to be the subalgebra~$\Diff(Q^\infty\Diff(S))$
of~$\Diff(S)$. 

The algebra $\Diff(\A)$ contains both~$S$ and~$\Der(\A)$. Now, it is
well-known that the algebra~$\Diff(S)$ is generated by its subset
$S\cup\Der(S)$, provided that the characteristic of the ground field is
zero. The purpose of this paper is to show that an analogous statement
holds for free arrangements:

\begin{Theorem-Intro}
If the characteristic of~$\kk$ is zero and the arrangement~$\A$ is free,
then the algebra~$\Diff(\A)$ is generated by $S\cup\Der(\A)$.
\end{Theorem-Intro}

We prove this by extending to differential operators of arbitrary order the
beautiful criterion given by Saito in~\cite{Saito} for deciding the
freeness of a divisor in terms of a certain Jacobian determinant ---given
for the case of hyperplane arrangements in~\cite{OT}. This theorem has been
proved by Francisco J. Calder\'on-Moreno in~\cite{CM} and later by Mathias
Schulze in~\cite{Schulze}, both working in the general context of an
arbitrary divisor in a complex manifold.

A natural question  ---analogous to the conjecture of Yoshikazu Nakai
\cite{Nakai} which asserts that an affine algebraic variety $X$ over a field of
characteristic zero  is smooth if its algebra of regular differential
operators is generated by functions~$\O_X$ and derivations~$\Der(\O_X)$---
is whether the necessary condition for freeness given by the theorem above
is also sufficient. This is not so: Schulze shows in~\cite{Schulze} that
the generic $3$-arrangement in~$\CC^3$ with equation $xyz(x+y+z)=0$ is a
counterexample: $\Diff(\A)$ is in this case generated by~$S$
and~$\Der(\A)$ and yet the $\A$ is not free.

\medskip

In the first section of the paper we prove a general lemma from multilinear
algebra, mostly for completeness. In the second one we present a `higher'
version of Jacobian determinants, and in the third, final one we use them
to prove the theorem.
We use the book \cite{OT} as our general reference for anything
related to hyperplane arrangements, and \cite{MR}*{Chapter 15}
and~\cite{Grothendieck} for algebras of differential operators.

\section{A lemma from multilinear algebra}

Let $R$ be a commutative ring, let $\ell$ be a positive integer and let
$M=(u_{i,j})$ be an $\ell\times\ell$ matrix with entries in~$R$. Let
$p\in\NN$ and let $W_p$ be the set of sequences $\idx{i}=(i_1,\dots,i_p)$
of integers of length~$p$ such that $1\leq i_1\leq\cdots\leq i_p\leq\ell$.
There are $\binom{p+\ell-1}{\ell-1}$ elements in~$W_p$.
If~$\idx{i}=(i_1,\dots,i_p)\in W_p$, we write $\idx{i}'$ the $\ell$-tuple
$(i'_1,\dots,i'_\ell)\in\NN_0^\ell$ with each entry~$i'_j$ denoting the
number of times the integer~$j$ appears in~$\idx{i}$ and put
$\idx{i'}!=i'_1!\cdots i'_\ell!$.

If $\idx{i}=(i_1,\dots,i_n)$ and $\idx{j}=(j_1,\dots,j_p)$ are two elements
of~$W_p$, we consider the permanent
  \[
  u_{\idx{i},\idx{j}} = \perm 
    \begin{pmatrix}
    u_{i_1,j_1} & \cdots & u_{i_1,j_p} \\
    \vdots      &        & \vdots      \\
    u_{i_p,j_1} & \cdots & u_{i_p,j_p}
    \end{pmatrix}
  \]
and let $M^{(p)}$ be the matrix $(u_{\idx{i},\idx{j}})_{\idx{i},\idx{j}\in
W_p}$, with rows and columns ordered according to a fixed but irrelevant
total ordering of the set~$W_p$.

\begin{Lemma}\label{lemma:multi}
We have $\det M^{(p)} = 
\gamma_{\ell,p} \cdot (\det M)^{\binom{p+\ell-1}{\ell}}$, with
  \( \displaystyle
  \gamma_{\ell,p} = \prod_{\idx{i}\in W_p}\idx{i}'!
  \).
\end{Lemma}

\begin{proof}
The identity we are to prove is the vanishing of a polynomial with integer
coefficients when evaluated at the entries of the matrix~$M$. A standard
argument shows that it is enough to prove it in the case where the ring~$R$
is the field~$\CC$ of complex numbers.

Let $\B=\{e_1,\dots,e_\ell\}$ be the standard basis of~$\CC^\ell$ and let
$\Sym(\CC^\ell)$ be the symmetric algebra on~$\CC^\ell$. If for each
$\idx{i}=(i_1,\dots,i_p)\in W_p$ we write $e_{\idx{i}}=e_{i_1}\cdots
e_{i_p}$, with the product being that of~$\Sym(\CC^\ell)$, then the set
$\B^{(p)}=\{e_{\idx{i}}:\idx{i}\in W_p\}$ is a basis of $\Sym^p(\CC^\ell)$,
the $p$th homogeneous component of~$\Sym(\CC^\ell)$. Let
$u:\CC^\ell\to\CC^\ell$ be the map such that $u(e_i)=\sum_{j=1}^\ell
u_{j,i}e_i$ for each $i\in\{1,\dots,\ell\}$ and let
  \[
  u^{(p)}:\Sym^p(\CC^\ell)\to\Sym^p(\CC^\ell)
  \]
be the map induced by~$u$
on~$\Sym^p(\CC^\ell)$. If~$\tilde M^{(p)}=(\tilde
u_{\idx{i},\idx{j}})_{\idx{i},\idx{j}\in W_p}$ is the matrix of~$u^{(p)}$
with respect to the basis~$\B^{(p)}$, so that
$u^{(p)}(e_{\idx{i}})=\sum_{j\in W_p}\tilde
u_{\idx{j},\idx{i}}e_{\idx{j}}$, then a straightforward computation shows
that for all $\idx{i}$,~$\idx{j}\in W_p$ we have
  \[ \label{eq:tuu}
  \idx{j}'!\,\tilde u^{(p)}_{\idx{i},\idx{j}}
    = u_{\idx{i},\idx{j}}.
  \]
It follows immediately from this that $\gamma_{\ell,p}\det\tilde
M^{(p)}=\det M^{(p)}$, and the claim of the lemma holds for~$M$ if and only
if we have
  \[
  \det\tilde M^{(p)} = (\det M)^{\binom{p+\ell-1}{\ell}}.
  \]
Both sides of this equality are obtained as the result of by evaluating a
continuous monoid homomorphism $M_\ell(\CC)\to\CC$ defined on the matrix
algebra~$M_\ell(\CC)$ at the matrix~$M$: as a consequence of this and of
the fact that $\Sym^p$ is a functor, we see that it it is enough to show
that the lemma holds if the matrix~$M$ is a diagonal complex matrix of the
form
  \[
  \begin{pmatrix}
    \lambda \\
    & 1 \\
    && \ddots \\
    &&& 1
  \end{pmatrix}.
  \]
We assume that this is the case. Let $\idx{i}$,~$\idx{j}\in W_p$. The matrix
  \[
    \begin{pmatrix}
    u_{i_1,j_1} & \cdots & u_{i_1,j_p} \\
    \vdots      &        & \vdots      \\
    u_{i_p,j_1} & \cdots & u_{i_p,j_p}
    \end{pmatrix}
  \]
is then a block diagonal matrix, with diagonal blocks of sizes $i'_1\times
j'_1$, \dots, $i'_\ell\times j'_\ell$, respectively. All entries in the
first of these blocks are equal to~$\lambda$ and all entries in all the
other $\ell-1$ of them are equal to~$1$. It is easy to see that the
permanent of this matrix, which we have named~$u_{\idx{i},\idx{j}}$, is
zero if~$\idx{i}\neq \idx{j}$, and when $\idx{i}=\idx{j}$ it is equal to
$\idx{i}'!\lambda^{i'_1}$. In view of~\eqref{eq:tuu}, we conclude that
  \[
  \tilde u_{\idx{i},\idx{j}}
    = \begin{cases*}
      \lambda^{i'_1}, & if $\idx{i}=\idx{j}$; \\
      0, & if not.
      \end{cases*}
  \]
The matrix~$\tilde M^{(p)}$ is thus a diagonal matrix and its determinant
is therefore
  \[
  \det\tilde M^{(p)} 
    = \prod_{\idx{i}\in W_p}\lambda^{i'_1}
    = \lambda^{\sum\limits_{\idx{i}\in W_p}i'_1}.
  \]
The exponent of~$\lambda$ in the last member of this chain of equalities is
equal to~$\binom{p+\ell-1}{\ell}$. As $\det M=\lambda$, the lemma is thus
proved.
\end{proof}

\section{Higher Jacobians}

If $u\in\Diff(S)$, for each $p\geq0$ and each choice of $f_1$,~\dots,~$f_p$
in~$S$ we define the \newterm{iterated commutator} $[u,f_1,\dots,f_p]$
recursively: if $p=0$ it is simply equal to $u$, and if $p>0$ we put
$[u,f_1,\dots,f_p]=[[u,f_1,\dots,f_{p-1}],f_p]$, with the outer brackets
denoting the usual commutator of the associative algebra~$\Diff(S)$. It is
easy to see that the expression~$[u,f_1,\dots,f_p]$ depends left $S$-linearly
on~$u$ and that it is a symmetric function of~$f_1$,~\dots,$f_p$.

We say that a differential operator~$u\in\Diff(S)$ has \newterm{order at
most~$p$} if for all $f_1$,~\dots,~$f_{p+1}\in S$ we have that
$[u,f_1,\dots,f_{p+1}]=0$ and we write $F_p\Diff(S)$ the subset
of~$\Diff(S)$ of all operators of order at most~$p$. We obtain in this way
an exhaustive increasing algebra filtration on~$\Diff(S)$, essentially by
definition of this algebra, and the associated graded algebra~$\gr\Diff(S)$
is commutative and, in fact, isomorphic to a polynomial algebra on~$2\ell$
variables.

If $u$ is a non-zero element of~$\Diff(S)$, there is a unique non-negative
integer~$p$ such that $u\in F_p\Diff(S)\setminus F_{p-1}\Diff(S)$, which we
call the \newterm{order} of~$u$, and the \newterm{principal symbol}
$\sigma(u)$ of~$u$ is the class $u+F_{p-1}\Diff(S)$ in the $p$th
homogeneous component $\gr_p\Diff(S)=F_p\Diff(S)/F_{p-1}\Diff(S)$ of the
graded algebra~$\gr\Diff(S)$. Since $F_0\Diff(S)=S$ and $F_{-1}\Diff(S)=0$,
we can identify $\gr_0\Diff(S)$ with~$S$, and with this in mind the
principal symbol of a polynomial~$f\in S\subseteq\Diff(S)$ is~$f$ itself.

\medskip

Let $p\geq1$. We will consider families $(u_{\idx{i}})_{\idx{i}\in W_p}$ of
elements of~$\Diff(S)$ indexed by~$W_p$ or, equivalently, elements
of~$\Diff(S)^{W_p}$. In particular, if
$\theta=(\theta_1,\dots,\theta_\ell)$ is an element of~$\Diff(S)^\ell$, we
will write $\theta^{(p)}$ the element $(u_{\idx{i}})_{\idx{i}\in W_p}$
of~$\Diff(S)^{W_p}$ such that $u_{\idx{i}}=\theta_{i_1}\cdots\theta_{i_p}$
for all $\idx{i}=(i_1,\dots,i_p)\in W_p$. On the other hand, if $u=
(u_{\idx{i}})_{\idx{i}\in W_p}$ is an element of~$\Diff(S)^{W_p}$ and
$w\in\Diff(S)$ and $\idx{j}\in W_p$, we will write $\subst{u}{w}{\idx{j}}$
the family $(v_{\idx{i}})_{\idx{i}\in W_p}\in\Diff(S)^{W_p}$ such that
$v_{\idx{j}}=w$ and $v_{\idx{i}}=u_{\idx{i}}$ for all $\idx{i}\in
W_p\setminus\{\idx{j}\}$: in other words, $\subst{u}{w}{\idx{j}}$ is
obtained form~$u$ by replacing the $\idx{j}$th component by~$w$.

If $f=(f_1,\dots,f_\ell)\in S^\ell$ and $u=(u_\idx{i})_{\idx{i}\in
W_p}\in\Diff(S)^{W_p}$, we define the \newterm{$p$th Jacobian of~$f$ with
respect to~$u$} to be 
  \[
  \frac{\partial^p f}{\partial^pu}
    = \det\bigl([u_{\idx{i}}, f_{j_1}, \dots, f_{j_p}](1)\bigr)_{\idx{i},\idx{j}\in W_p},
  \]
which is an element of~$S$. If $p=1$, then we can identify $W_1$ with
$\{1,\dots,\ell\}$ in an obvious way, so that an element~$u$
of~$\Diff(S)^{W_1}$ is just a $p$-tuple $(u_1,\dots,u_p)$ of elements
of~$\Diff(S)$. If these happen to be elements of~$\Der(S)$, then for all
$f\in S$ we have $[u_i,f](1)=u_i(f)$, and then for all
$f=(f_1,\dots,f_\ell)\in S^\ell$ we have that
  \[
  \frac{\partial^1f}{\partial^1u}
    = \begin{vmatrix}
        u_1(f_1)    & \cdots & u_1(f_\ell)    \\
        \vdots      &        & \vdots         \\
        u_\ell(f_1) & \cdots & u_\ell(f_\ell) 
      \end{vmatrix},
  \]
the usual Jacobian of the functions $f_1$,~\dots,~$f_\ell$ with respect to
the derivations $u_1$,~\dots,~$u_\ell$: this explains the name.

It is easy to see that the Jacobian $\partial^pf/\partial^pu$ depends
$S$-linearly on each component of~$u$ and, using Lemma~\ref{lemma:multi},
that whenever $A=(a_{i,j})$ is an $\ell\times\ell$ matrix with entries
in~$\kk$ we have that
  \[
  \frac{\partial^p (Af)}{\partial^pu}
    = (\det A)^{\binom{p+\ell-1}{\ell}}\cdot \frac{\partial^p f}{\partial^pu}.
  \]
More significantly, we have:

\begin{Lemma}
\begin{thmlist}\fixhspace

\item If any of the entries of~$u\in\Diff(S)^{W_p}$ has order at most
$p-1$, then we have that $\partial^pf/\partial^p u=0$.

\item If $\theta=(\theta_1,\dots,\theta_\ell)\in S^\ell$ has all its
components of order at most~$1$, then
  \[
  \frac{\partial^pf}{\partial^p\theta^{(p)}}
    = \gamma_{\ell,p}
      \left(
      \frac{\partial^1 f}{\partial^1\theta}
      \right)^{\binom{p+\ell-1}{\ell}},
  \]
with $\gamma_{\ell,p}$ the integer defined in Lemma~\ref{lemma:multi}.
\end{thmlist}
\end{Lemma}

\begin{proof}
If $\idx{i}\in W_p$ is such that $u_{\idx{i}}\in F_{p-1}\Diff(S)$, then
$[u_\idx{i},f_{j_1},\dots,f_{j_p}]=0$ for all $\idx{j}=(j_1,\dots,j_p)\in
W_p$, so that the matrix whose determinant is~$\partial^pf/\partial^pu$ has
a zero column. Of course, the first claim of the lemma follows from this.

To prove the second one, and in view of Lemma~\ref{lemma:multi}, it is
enough that we fix $\idx{i}$,~$\idx{j}\in W_p$ and show that
  \[ \label{eq:jac}
  [\theta_{i_1}\cdots\theta_{i_p},f_{j_1},\dots,f_{j_p}](1)
    = \perm
      \begin{pmatrix}
      [\theta_{i_1},f_{j_1}](1) & \cdots & [\theta_{i_1},f_{j_p}](1) \\
      \vdots                 &        & \vdots                 \\
      [\theta_{i_p},f_{j_1}](1) & \cdots & [\theta_{i_p},f_{j_p}](1) \\
      \end{pmatrix}.
  \]
We have that in~$\Diff(S)$
  \begin{align}\MoveEqLeft{}
  [\theta_{i_1}\cdots\theta_{i_p},f_{j_1},\dots,f_{j_p}]
     = [[\theta_{i_1}\cdots\theta_{i_p},f_{j_1}],f_2,\dots,f_{j_p}] \\
    &= \sum_{k=1}^p
       [\theta_{i_1}
        \cdots\theta_{i_{k-1}}
        [\theta_{i_k},f_1]
        \theta_{i_{k+1}}
        \cdots\theta_{i_p}, f_2,\dots,f_p]. \label{eq:sum}
  \end{align}
For each $k\in\{1,\dots,p\}$ the operators
  \begin{gather}
  \theta_{i_1} \cdots\theta_{i_{k-1}} [\theta_{i_k},f_1] \theta_{i_{k+1}} 
                \cdots\theta_{i_p}
\shortintertext{and}
  [\theta_{i_k},f_1]\theta_{i_1} \cdots\theta_{i_{k-1}}  \theta_{i_{k+1}} 
                \cdots\theta_{i_p}
  \end{gather}
differ by an element of~$F_{p-2}\Diff(S)$, so the sum~\eqref{eq:sum} can be
rewritten as
  \[
  \sum_{k=1}^p
       [[\theta_{i_k},f_1]
        \theta_{i_1}
        \cdots\hat\theta_{i_k}
        \cdots\theta_{i_p}, f_2,\dots,f_p],
  \]
with the accent denoting omission of the marked factor, as usual.
As $[\theta_{i_k},f_1]\in S$ for all $k\in\{1,\dots,p\}$,
the value at~$1\in S$ of this operator is
  \[
  \sum_{k=1}^p
       [\theta_{i_k},f_1](1)\cdot
       [\theta_{i_1}
        \cdots\hat\theta_{i_k}
        \cdots\theta_{i_p}, f_2,\dots,f_p](1).
  \]
Inductively, then, we know that this is equal to
  \[
  \sum_{k=1}^p
       [\theta_{i_k},f_1](1)
       \cdot
      \perm
      \begin{pmatrix}
      [\theta_{i_1},f_{j_2}](1) 
        & \cdots 
        & [\theta_{i_1},f_{j_p}](1) \\
      \vdots                 
        &        
        & \vdots                 \\
      \widehat{[\theta_{i_k},f_{j_2}](1)} 
        & \cdots 
        & \widehat{[\theta_{i_k},f_{j_p}](1)} \\
      \vdots                 
        &        
        & \vdots                 \\
      [\theta_{i_p},f_{j_2}](1) 
        & \cdots 
        & [\theta_{i_p},f_{j_p}](1) \\
      \end{pmatrix},
  \]
which is the Laplace expansion for the permanent that appears on the right
in~\eqref{eq:jac} along the first column.
\end{proof}

\section{Differential operators tangent to a hyperplane arrangement}

As in the introduction, let $\A=\{H_1,\dots,H_r\}$ be a central hyperplane
arrangement in the vector space~$V$, let $\alpha_1$,~\dots,~$\alpha_r\in
V^*$ be such that $H_i=\ker\alpha_i$ for all $i\in\{1,\dots,r\}$ and let
$Q=\prod_{i=1}^r\alpha_i$ be a defining polynomial for~$\A$. For each
$i\in\{1,\dots,r\}$ we let $\beta_i=Q/\alpha_i$, so that
$Q=\alpha_i\beta_i$ and $\alpha_i$ does not divide~$\beta_i$.

For each right ideal~$I$ in~$\Diff(S)$ we write $\II(I)$ the
\newterm{idealizer} if~$I$ in~$\Diff(S)$, that is, the subalgebra
  \[
  \II(I) = \{ u\in\Diff(S) : uI\subseteq I \},
  \]
which is the largest subalgebra of~$\Diff(S)$ which contains~$I$ as a
bilateral ideal. If $f\in\Diff(S)$, we write
  \[
  \II(f^\infty\Diff(S)) = \bigcap_{t\geq1}\II(f^t\Diff(S)).
  \]
The \newterm{algebra~$\Diff(\A)$ of differential operators tangent to~$\A$}
is the subalgebra $\Diff(Q^\infty\Diff(S))$ of~$\Diff(S)$. We want to give
a slightly more convenient description of~$\Diff(\A)$, and to do that we
start with the following simple observation:

\begin{Lemma}\label{lemma:fact}
Let $\alpha$ and~$\beta$ be elements of~$S$ without a non-constant common
factor. If $u\in\Diff(S)$, then
  \[
  u\beta\in\alpha\Diff(S) \implies u\in\alpha\Diff(S).
  \]
\end{Lemma}

\begin{proof}
Let $u\in\Diff(S)$ and suppose that there exists a $v\in\Diff(S)$ such that
  \[ \label{eq:ub}
  u\beta=\alpha v.
  \]
If $u$ is zero then of course $u\in\alpha\Diff(S)$, so we suppose it is
not. We may then consider the order~$p$ of~$u$, so that $u\in
F_p\Diff(S)\setminus F_{p-1}\Diff(S)$, and proceed by induction on~$p$.

Suppose first that $p=0$, so that in fact $u\in S$. It follows from
equation~\eqref{eq:ub} that we necessarily have that $v\in S$ and, since
$\alpha$ does not divide~$\beta$ in~$S$, that $\alpha$ divides~$u$: in
particular, we have that $u\in\alpha S\subseteq\alpha\Diff(S)$, as we want.

Suppose next that $p>0$. From~\eqref{eq:ub} we see that $v\in
F_p\Diff(S)\setminus F_{p-1}\Diff(S)$ and that
$\sigma(u)\beta=\alpha\sigma(v)$ in the associated graded
algebra~$\gr\Diff(S)$. As $\alpha$ does not divide~$\beta$, we see that
$\alpha$ divides~$\sigma(u)$ and, therefore, that there exists an operator $w\in
F_p\Diff(S)\setminus F_{p-1}\Diff(S)$ of order~$p$ such that $u' = u - \alpha w \in
F_{p-1}\Diff(S)$. Now $u'\beta=u\beta-\alpha
w\beta=\alpha(v-w\beta)\in\alpha\Diff(S)$ and $u'$ has order at most $p-1$,
so that inductively we know that $u'\in\alpha\Diff(S)$. We can then conclude
that $u=\alpha w+u'\in\alpha\Diff(S)$, completing the induction.
\end{proof}

Using this lemma, we can prove that $\Diff(\A)$ is the intersection of the
infinite idealizers of the defining equations of the hyperplanes of the
arrangement, much as the Lie algebra~$\Der(\A)$ is equal to
$\bigcap_{i=1}^r\Der(\alpha_iS)$; see~\cite{OT}*{Proposition 4.8}.

\begin{Proposition}
We have $\Diff(\A)=\bigcap_{i=1}^r\II(\alpha_i^\infty\Diff(S))$.
\end{Proposition}

\begin{proof}
Let $u$ be an element of that intersection, so that for each
$i\in\{1,\dots,r\}$ and $t\geq1$ there is a $u_{i,t}\in\Diff(S)$ such that
$u\alpha_i^t=\alpha_i^tu_{i,t}$. 

Let $t\geq1$. If $i\in\{1,\dots,r\}$, then we have that
  \[
  uQ^t = u\alpha_i^t\beta_i^y = \alpha_i^i u_{i,t}\beta_i^t \in \alpha_i^t\Diff(S),
  \]
and therefore
  \[ \label{eq:dd}
  uQ^t \in \bigcap_{i=1}^r \alpha_i^t\Diff(S) = Q^t\Diff(S),
  \]
so that $u$ is in~$\II(Q^t\Diff(S))$. The equality in~\eqref{eq:dd} follows
immediately from the fact that $\Diff(S)$ is free as a left module over its
subalgebra~$S$ and in~$S$ we have that $\bigcap_{i=1}^r\alpha_i^tS=Q^tS$.
Of course, this implies that $u\in\Diff(\A)$.

To prove the reverse containment, let $u\in\Diff(S)$ be an element
of~$\Diff(\A)$, so that for each $t\geq1$ there is a $v_t\in\Diff(S)$ such
that $uQ^t=Q^tv_t$. If $i\in\{1,\dots,r\}$ and $t\geq1$, then we have that
  \[
  u\alpha_i^t\beta_i^t 
        = uQ^t = Q^tv_t = \alpha_i^t\beta_i^tv_t \in \alpha_i^t\Diff(S)
  \]
and we may conclude that $u\alpha_i^t\in\alpha_i^t\Diff(S)$ using
Lemma~\ref{lemma:fact}. We see in this way that $u$ is in
$\II(\alpha_i^\infty\Diff(S))$ and, therefore, in the intersection
described in the proposition. 
\end{proof}

\begin{Lemma}\label{lemma:mult}
Let $u\in\Diff(\A)$ and let $i\in\{1,\dots,r\}$.
\begin{thmlist}

\item For all $t\geq1$ we have $u(\alpha_i^tS)\subseteq\alpha_i^tS$.

\item Let $p$,~$q\in\NN_0$ be such that $0\leq q\leq p$. If at least $q$
components of of the $p$-tuple $(f_1,\dots,f_p)\in S^p$ are equal
to~$\alpha_i$, then $[u,f_1,\dots,f_p](1)\in\alpha_i^q S$.

\end{thmlist}
\end{Lemma}

\begin{proof}
\thmitem{1} Let $t\geq1$. As $u\in\Diff(\A)$, there is a $v\in\Diff(S)$
such that $u\alpha_i^t=\alpha_i^tv$, and then for all $f\in S$ we have
  \[
  u(\alpha_i^t f)
        = (u\alpha_i^t)(f)
        = (\alpha_i^tv)(f)
        = \alpha_i^t v(f) 
        \in \alpha_i^tS.
  \]

\thmitem{2} Suppose that at least $q$ components of $(f_1,\dots,f_p)\in
S^p$ are equal to~$\alpha_i$ and let $I=\{1,\dots,p\}$. One sees easily
that
  \[
  [u,f_1,\dots,f_p](1) =
     \sum_{J\subseteq I}
     (-1)^{\abs{J}}
     \prod_{j\in J}f_{j}
     \cdot
     u\left(\prod_{k\in I\setminus J} f_{k}\right).
  \]
The first part of the lemma implies that for each $J\subseteq I$ the $J$th
summand of this sum is in $\alpha_i^qS$, so that so is the sum.
\end{proof}

\begin{Proposition}\label{prop:jdiffa}
Let $f=(x_1,\dots,x_\ell)$ be an ordered basis of~$V^*$ and let $p\in\NN$.
If $(u_{\idx{i}})_{\idx{i}\in W_p}$ is a family of elements of~$\Diff(\A)$
indexed by~$W_p$, then 
  \[
  \frac{\partial^pf}{\partial^p u} \in Q^{\binom{p+\ell-1}{\ell}}S.
  \]
\end{Proposition}

\begin{proof}
Let $k\in\{1,\dots,r\}$ and let $a_1$,~\dots,~$a_\ell\in\kk$ be such that
$\alpha_k=\sum_{j=1}^\ell a_jx_j$. Without loss of generality, we can
assume that $a_1\neq0$ and then, of course, $g=(\alpha_k,x_2,\dots,x_\ell)$
is another ordered basis of~$V^*$, whose elements we relabel
$y_1$,~\dots,~$y_\ell$. If $\idx{i}$,~$\idx{j}\in W_p$, then $\alpha_k$
appears $j'_1$ times in the $p$-tuple $(y_{j_1},\dots,y_{j_p})$ and the
second part of Lemma~\ref{lemma:mult} tells us that
  \[
  [u_{\idx{i}},y_{j_1},\dots,y_{j_p}](1) 
        \in \alpha_k^{j'_1} S.
  \]
This implies that $[u_{\idx{i}},y_{j_1},\dots,y_{j_p}](1)\in\bigcap_{k=1}^r
\alpha_k^{j'_1}S=Q^{j'_1}S$, so that the $\idx{j}$th column of the matrix
whose determinant is the Jacobian $\partial^p g/\partial^p u$ is divisible
by~$Q^{j'_1}$. Of course, it follows from this that the Jacobian itself is
divisible by
  \[
  \prod_{\idx{j}\in W_p}Q^{j'_1} = Q^{\binom{p+\ell-1}{\ell}}.
  \]
This is what we set out to prove.
\end{proof}

Let us denote $\Delta(\A)$ the subalgebra of~$\Diff(S)$ generated by~$S$
and~$\Der(\A)$.

\begin{Proposition}
The subalgebra~$\Delta(\A)$ is contained in~$\Diff(\A)$. If
$\{\theta_1,\dots,\theta_m\}$ is a subset of~$\Der(\A)$ that generates it
as a left $S$-module, then the set 
  \[
  \{\theta_{i_1}\dots\theta_{i_q} : 0\leq q\leq p, 1\leq i_1\leq\cdots\leq i_q\leq m\}
  \]
generates $\Delta(\A)\cap F_p\Diff(S)$ as a left $S$-module.
\end{Proposition}

\begin{proof}
To prove the first part, it is enough to show that $S$ and~$\Der(\A)$ are
contained in~$\Diff(\A)$, and it is obvious that $S$ is. Let $\delta$ be an
element of~$\Der(\A)$, so that $\delta(Q)=Qf$ for some $f\in S$. If
$t\geq1$, in~$\Diff(S)$ we have that
  \begin{align}
  \delta Q^t 
       &= Q^t\delta + [\delta,Q^t]
        = Q^t\delta + \delta(Q^t)
        = Q^t\delta + tQ^{t-1}\delta(Q) \\
       &= Q^t(\delta+tf)
        \in Q^t\Diff(S),
  \end{align}
so that $\delta\in\II(Q^t\Diff(S))$ and, therefore,
$\delta\in\II(Q^\infty\Diff(S))=\Diff(\A)$.

The second part of the proposition, on the other hand, is a direct
consequence of the commutation relation $\theta_i h=h\theta_i+\theta_i(h)$
that holds in~$\Diff(S)$ for all $i\in\{1,\dots,\ell\}$ and all $h\in S$.
\end{proof}

It follows immediately from the definitions that
$Q\Der(S)\subseteq\Der(\A)$. It is easy to extend that observation to
differential operators of all orders:

\begin{Proposition}\label{prop:transport}
If $u\in F_p\Diff(S)$, then $Q^{\binom{p+1}{2}}u\in\Delta(\A)$.
\end{Proposition}

\begin{proof}
Let $(x_1,\dots,x_\ell)$ be an ordered basis of~$V^*$ and let
$(\partial_1,\dots,\partial_\ell)$ be the corresponding dual ordered basis
of~$\Der(S)$. Let $p\geq0$ and $u\in F_p\Diff(S)$. There are a
$u=(u_{\idx{i}})_{\idx{i}\in W_p}\in S^{W_p}$ and a $v\in F_{p-1}\Diff(S)$
such that 
  \[
  u=\sum_{\idx{i}\in W_p}u_{\idx{i}}\partial_{i_1}\cdots\partial_{i_p}+v.
  \]
As $Q\Der(S)\subseteq\Der(\A)$ and the associated graded algebra
$\gr\Diff(S)$ is commutative, there is a $v'\in F_{p-1}\Diff(S)$ such that
  \[
  Q^pu = \sum_{\idx{i}\in W_p}
        u_{\idx{i}}(Q\partial_{i_1})\cdots(Q\partial_{i_p})
        + v'.
  \]
The sum appearing in the right hand side of this equality is
in~$\Delta(\A)$ and we may assume inductively that
$Q^{\binom{p}{2}}v'\in\Delta(\A)$, so that
  \[
  Q^{\binom{p+1}{2}}u
        = Q^{p+\binom{p}{2}}u
        = Q^{\binom{p}{2}}\sum_{\idx{i}\in W_p}
                u_{\idx{i}}(Q\partial_{i_1})\cdots(Q\partial_{i_p})
           + Q^{\binom{p}{2}}v'
        \in \Delta(\A),
  \]
as we want.
\end{proof}

We are finally in position to prove the main result of this paper:

\begin{Theorem}
Let $\A$ be a free arrangement of hyperplanes in~$V$ and suppose the ground
field~$\kk$ has characteristic zero. The algebra $\Diff(\A)$ is generated
by~$S$ and~$\Der(\A)$.
\end{Theorem}

\begin{proof}
Let $f=(x_1,\dots,x_\ell)$ be an ordered basis of~$V^*$ and let
$\theta=(\theta_1,\dots,\theta_\ell)$ be an homogeneous basis of the
$S$-module~$\Der(\A)$. Saito's criterion for freeness \cite{OT}*{Theorem
4.19} tells us that
  \[
  \frac{\partial^1f}{\partial^1\theta} = \lambda\, Q
  \]
for some non-zero scalar $\lambda\in\kk$. 

Let $u\in\Diff(\A)$ be a non-zero element which has order at most~$p$. In
view of Proposition~\ref{prop:transport}, we have
$Q^{\binom{p+1}{2}}u\in\Delta(\A)$, and then there exist a family
$(u_{\idx{i}})_{\idx{i}\in W_p}\in S^{W_p}$ of polynomials indexed by~$W_p$
and a $v\in \Delta(\A)\cap F_{p-1}\Diff(S)$ such that
  \begin{equation} \label{eq:q0}
  Q^{\binom{p+1}{2}}u = 
        \sum_{\idx{i}\in W_p} u_{\idx{i}}\theta_{i_1}\cdots\theta_{i_p}
        + v.
  \end{equation}
Let us fix $\idx{k}\in W_p$. We have
  \[
  \frac{\partial^p f}{\partial^p\subst{\theta^{(p)}}{Q^{\binom{p+1}{2}}u}{\idx{k}}}
    = \sum_{\idx{i}\in W_p}
        u_{\idx{i}}
        \frac{\partial^p f}
             {\partial^p\subst{\theta^{(p)}}{\theta_{i_1}\cdots\theta_{i_p}}{\idx{k}}}
        +
        \frac{\partial^p f}{\partial^p\subst{\theta^{(p)}}{v}{\idx{k}}}.
  \]
The second summand in the right hand side of this equality vanishes, because
$v$~has order at most $p-1$, and only possibly non-zero term in the sum is
the one in which $\idx{i}$ coincides with~$\idx{k}$: it follows then that
  \begin{equation} \label{eq:qa}
  \begin{split}
  \frac{\partial^p f}{\partial^p\subst{\theta^{(p)}}{Q^{\binom{p+1}{2}}u}{\idx{k}}}
       &= u_{\idx{k}}\,
         \frac{\partial^p f}{\partial^p\theta^{(p)}}
        = \gamma_{\ell,p}u_{\idx{k}} 
                \left(
                \frac{\partial^1f}{\partial^1\theta}
                \right)^{\binom{p+\ell-1}{\ell}} \\
       &= \gamma_{\ell,p}u_{\idx{k}} (\lambda Q)^{\binom{p+\ell-1}{l}}.
  \end{split}
  \end{equation}
On the other hand, the $S$-linearity of higher Jacobians and
Proposition~\ref{prop:jdiffa} imply that
  \begin{equation} \label{eq:qb}
  \frac{\partial^p f}{\partial^p\subst{\theta^{(p)}}{Q^{\binom{p+1}{2}}u}{\idx{k}}}
    = Q^{\binom{p+1}{2}} \frac{\partial^p f}{\partial^p\subst{\theta^{(p)}}{u}{\idx{k}}}
    \in Q^{\binom{p+1}{2}+\binom{p+\ell-1}{\ell}} S.
  \end{equation}
As our ground field~$\kk$ is of characteristic zero, the scalar
$\gamma_{\ell,p}$ is a unit and comparing the right hand sides of the
equalities~\eqref{eq:qa} and~\eqref{eq:qb} we see that $u_{\idx{k}}$ is
divisible by~$Q^{\binom{p+1}{2}}$. There is then a family $(\bar
u_{\idx{i}})\in S^{W_p}$ such that $u_{\idx{i}}=Q^{\binom{p+1}{2}}\bar
u_{\idx{i}}$ for all $\idx{i}\in W_p$. Going back to~\eqref{eq:q0}, we see
that
  \[
  Q^{\binom{p+1}{2}}\left(
        u -
        \sum_{\idx{i}\in W_p}\bar u_{\idx{i}}\theta_{i_1}\cdots\theta_{i_p}
        \right)
      = v \in F_{p-1}\Diff(S).
  \]
The difference
  \[
  u' = u -
        \sum_{\idx{i}\in W_p}\bar u_{\idx{i}}\theta_{i_1}\cdots\theta_{i_p}
  \]
is therefore an element of~$\Diff(\A)$ of order at most~$p-1$ and,
proceeding by induction, we know that it belongs to~$\Delta(\A)$: this
implies, of course, that so does~$u$. We conclude in this way that
$\Diff(\A)\subseteq\Delta(\A)$, as we wanted.
\end{proof}


\end{document}